\documentclass[a4paper,10pt]{article}

\usepackage[utf8]{inputenc}
\usepackage[T1]{fontenc}
\usepackage{amsmath,amssymb}
\usepackage{amsthm}
\usepackage[english]{babel}
\usepackage{enumerate}
\usepackage{float}
\usepackage{url}

\renewcommand{\bar}{\overline}

\DeclareMathOperator{\im}{Im}

\newcommand{\C}{\mathbb C}
\newcommand{\HH}{\mathbb H}
\newcommand{\Q}{\mathbb Q}
\newcommand{\Z}{\mathbb Z}

\newcommand{\OO}{\mathcal O}

\theoremstyle{plain}
\newtheorem{theorem}{Theorem}[section]
\newtheorem{lemma}[theorem]{Lemma}
\newtheorem{proposition}[theorem]{Proposition}

\title{Linear independence of powers of singular moduli of degree $3$}
\author{Florian Luca, Antonin Riffaut}
\date{\today}

\begin{document}

\maketitle

\begin{abstract}
We show that two distinct singular moduli $j(\tau),j(\tau')$, such that for some positive integers $m, n$ the numbers $1,j(\tau)^m$ and $j(\tau')^n$ are linearly dependent over $\mathbb{Q}$ generate the same number field of degree at most $2$. This completes a result of Riffaut, who proved the above theorem except for two explicit pair of exceptions consisting of numbers of degree $3$. The purpose of this article is to treat these two remaining cases.
\end{abstract}

\section{Introduction}

Let~$j$ be the classical $j$-function on the Poincaré plane ${\HH=\{z\in \C:\im z>0\}}$. 
A \textsl{singular modulus} is a number of the form $j(\tau)$, where ${\tau\in \HH}$ is a complex algebraic number of degree~$2$. It is known that $j(\tau)$ is an algebraic integer and Class Field Theory tells that
\[[\Q(j(\tau)):\Q]=[\Q(\tau,j(\tau)):\Q(\tau)]=h_\Delta\]
is the class number of the order ${\OO_\Delta=\Z[(\Delta+\sqrt\Delta)/2]}$, where~$\Delta$ is the discriminant of the minimal polynomial of~$\tau$ over~$\Z$. Moreover, $\Q(\tau,j(\tau))/\Q(\tau)$ is an abelian Galois extension with Galois group (canonically) isomorphic to the class group of the order $\OO_\Delta$. One can also interpret $\OO_\Delta$ as the automorphism ring of the lattice $\langle1,\tau\rangle$, or of the corresponding elliptic curve. For all details, see, for instance, \cite[§7 and §11]{Co89}.

Starting from the ground-breaking article of André \cite{An98}, equations involving singular moduli were studied by many authors, see \cite{Al15,Yu16,Ri17} for a historical account and further references. In particular, Kühne \cite{Ku13} proved that the equation $x+y=1$ has no solutions in singular moduli $x,y$, and Bilu et al. \cite{Yu13} proved the same conclusion holds for the equation $xy=1$. These results were generalized in \cite{Al15} and \cite{Yu16}. In \cite{Al15}, solutions of all linear equations $Ax+By=C$, with $A,B,C\in\mathbb{Q}$, were determined. Here is the main result of \cite{Al15}.

\begin{theorem}[Allombert et al. \cite{Al15}]
\label{th:allombert}
Let $x,y$ be two singular moduli, and $A,B,C$ rational numbers with $AB\neq 0$. Assume that $Ax+By=C$. Then we have one of the following options:
\begin{description}
\item[(trivial case)] $A+B=C=0$ and $x=y$;
\item[(rational case)] $x,y\in\mathbb{Q}$;
\item[(quadratic case)] $x\neq y$ and $x,y$ generate the same number field over $\mathbb{Q}$ of degree $2$.
\end{description}
\end{theorem}

This result is best possible, since in both the rational case and the quadratic case of Theorem \ref{th:allombert}, one easily finds $A,B,C\in\mathbb{Q}$ such that $AB\neq 0$ and $Ax+By=C$. Moreover, the lists of singular moduli of degrees $1$ and $2$ over $\mathbb{Q}$ are widely available or can be easily generated using a suitable computer package, like \textsf{PARI} \cite{Pari}. In particular, there are $13$ rational singular moduli, and $29$ pairs of $\mathbb{Q}$-conjugate singular moduli of degree $2$; see \cite[Section 1]{Yu16} for more details. This means that Theorem \ref{th:allombert} gives a completely explicit characterization of all solutions.

In \cite{Ri17}, Riffaut generalized Theorem \ref{th:allombert} by introducing exponents; that is, instead of equation $Ax+By=C$, he considered the more general equation ${Ax^m+By^n=C}$, where the positive integer exponents $m,n$ are unknown as well. He proved that, if $x\neq y$, then $x,y$ generate the same number field of degree $h\leq 3$, and $h=3$ is possible only if either $\{\Delta,\Delta'\}=\{-4\cdot 23,-23\}$, or $\{\Delta,\Delta'\}=\{-4\cdot 31,-31\}$, where $\Delta,\Delta'$ denote the respective discriminants of $x$ and $y$. In this article, we eliminate these two remaining cases. Here is the statement of our result.

\begin{theorem}
\label{th:main}
Let $x=j(\tau),y=j(\tau')$ be two singular moduli of respective discriminants $\Delta$ and $\Delta'$, and $m,n$ two positive integers. If $\{\Delta,\Delta'\}={\{-4\cdot 23,-23\}}$ or $\{\Delta,\Delta'\}=\{-4\cdot 31,-31\}$, then the numbers $1,x^m,y^n$ are linearly independent over $\mathbb{Q}$.
\end{theorem}

Consequently, Theorem \ref{th:main} together with \cite[Theorem 1.5]{Ri17} completely solve the above equation for distinct singular moduli, and we deduce the following Theorem.

\begin{theorem}
Let $x=j(\tau),y=j(\tau')$ be two distinct singular moduli of respective discriminants $\Delta$ and $\Delta'$, and $m,n$ two positive integers. Assume that $Ax^m+By^n=C$, for some $A,B,C\in\mathbb{Q}^{\times}$. Then $x$ and $y$ generate the same number field over $\mathbb{Q}$ of degree at most $2$.
\end{theorem}

As previously, this result is now best possible for distinct singular moduli, since if $h\leq 2$, then for all exponents $m,n$, one easily finds $A,B,C\in\mathbb{Q}^{\times}$ such that $Ax^m+By^n=C$. However, our current methods are still not able to handle the case $x=y$, which is equivalent to the following question: can a singular modulus of degree $3$ or higher be a root of a trinomial with rational coefficients? Much about trinomials is known, but this knowledge is still insufficient to rule out such a possibility. Otherwise, the assumption $C\neq 0$ is seemingly restrictive, but in fact, the case $C=0$ is contained in \cite[Theorem 1.6]{Ri17}.

Our calculations were performed using the \textsf{PARI/GP} package \cite{Pari}. The sources are available from the second author.

\section{Preliminaries}

Below we briefly recall some basic facts about the conjugates of a singular modulus and the height of an algebraic number.

\paragraph{Fields generated by a power of a singular modulus}~

\medskip

Let $j(\tau)$ be a singular modulus of discriminant $\Delta$. It is well-known that the conjugates of $j(\tau)$ over $\mathbb{Q}$ can be described explicitly; see, for instance, \cite[Subsection 2.2]{Ri17}. In particular, $j(\tau)$ admits one real conjugate which has the property that it is much larger in absolute value than all its other conjugates, called the \emph{dominant $j$-value} of discriminant $\Delta$. As a useful consequence, a singular modulus and any of its powers generate the same field over $\mathbb{Q}$; see \cite[Lemma 2.6]{Ri17}, a statement which we reproduce below.

\begin{lemma}
\label{lemma:power-j}
Let $x$ be a singular modulus of discriminant $\Delta$, with $|\Delta|\geq 11$, and $n$ a non-zero integer. Then $\mathbb{Q}(x)=\mathbb{Q}(x^n)$.
\end{lemma}

\paragraph{The height of a non-zero algebraic number}~

\medskip

\sloppypar{Let $\alpha$ be a non-zero algebraic number of degree $d$ over $\mathbb{Q}$, and $\alpha_1=\alpha,\alpha_2,\dots,\alpha_d$ all its conjugates in $\bar{\mathbb{Q}}$. The logarithmic height of $\alpha$, denoted by $\mathrm{h}(\alpha)$, is defined to be}
\[\mathrm{h}(\alpha)=\frac{1}{d}\left(\log|a|+\sum_{k=1}^d\log\max\{1,|\alpha_k|\}\right),\]
where $a$ is the leading coefficient of the minimal polynomial of $\alpha$ in $\mathbb{Z}$. In particular, $\log|a|=0$ when $\alpha$ is an algebraic integer.

Here are some useful properties of the logarithmic height.
\begin{itemize}
\item For any non-zero algebraic number $\alpha$ and $\lambda\in\mathbb{Q^*}$, we have $\mathrm{h}(\alpha^{\lambda})=|\lambda|\mathrm{h}(\alpha)$. In particular, $\mathrm{h}(1/\alpha)=\mathrm{h}(\alpha)$. See \cite[Lemma 1.5.18]{Bo06}.
\item For any two non-zero algebraic numbers $\alpha$ and $\beta$, we have $\mathrm{h}(\alpha\beta)\leq\mathrm{h}(\alpha)+\mathrm{h}(\beta)$.
\end{itemize}

\section{Linear forms in two logarithms}

Let $\alpha$ be an algebraic number with $|\alpha|=1$ but not a root of unity and $m$ a positive integer. We are interested in estimating the quantity $\lambda=1-\alpha^n$, which is closely related to a linear form in two logarithms.

Laurent, Mignotte and Nesterenko describe in \cite{La95} a lower bound on the absolute value of a general linear form in two logarithms, see \cite[Théorème 3]{La95}. In our particular case, Mignotte give in \cite{Yu01} a slight sharpening of this bound. The following Theorem is a corollary of \cite[Theorems A.1.2 and A.1.3]{Yu01}.

\begin{theorem}
\label{th:mignotte}
Let $\alpha$ be a complex algebraic number with $|\alpha|=1$, but not a root of unity, and $m$ a positive integer. There exists an effective computable constant $c_1(\alpha)>0$, depending only on the degree $d$ of $\alpha$ over $\mathbb{Q}$ and its logarithmic height $\mathrm{h}(\alpha)$, such that
\[|1-\alpha^m|>0.99\mathrm{e}^{-c_1(\alpha)(\log m)^2}.\]
\end{theorem}

\begin{proof}
We briefly detail the proof, especially to explain how to compute $c_1(\alpha)$ in terms of $d$ and $\mathrm{h}(\alpha)$.

We apply \cite[Theorems A.1.2 and A.1.3]{Yu01} to the linear form
\[\Lambda=2i\pi-m\log\alpha,\]
where we choose the principal complex logarithm (defined on $\mathbb{C}\setminus\mathbb{R}^-$) for $\log\alpha$. We have
\[\log|\Lambda|>-(9.03\mathcal{H}^2+0.23)(D\mathrm{h}(\alpha)+25.84)-2\mathcal{H}-2\log\mathcal{H}-0.7D+2.07,\]
where $D=d/2$ and $\mathcal{H}=D(\log m-0.96)+4.49\leq c_1'(d)\log m$ for $m\geq 13$, with
\[c_1'(d)=D+\max\left\{0,\frac{4.49-0.96D}{\log 13}\right\}>0.\]
Hence,
\begin{multline*}
\log|\Lambda|>-(\log m)^2\left(9.03c_1'(d)^2(D\mathrm{h}(\alpha)+25.84)+\frac{2c_1'(d)}{\log m}+\frac{2\log\log m}{(\log m)^2}\right.\\
+\left.\frac{0.23(D\mathrm{h}(\alpha)+25.84)+2\log c_1'(d)+0.7D-2.07}{(\log m)^2}\right)>-c_1(\alpha)(\log m)^2,
\end{multline*}
with
\begin{align*}
c_1(\alpha)= &~ 9.03c_1'(d)^2(D\mathrm{h}(\alpha)+25.84)+\frac{2c_1'(d)}{\log 13}+\frac{2\log\log 13}{(\log 13)^2}\\
&+\frac{0.23(D\mathrm{h}(\alpha)+25.84)+2\log c_1'(d)+0.7D-2.07}{(\log 13)^2}.
\end{align*}
It follows that
\[|1-\alpha^m|>\frac{\mathrm{e}^{-c_1(\alpha)(\log m)^2}}{1+\mathrm{e}^{-c_1(\alpha)(\log m)^2}}>0.99\mathrm{e}^{-c_1(\alpha)(\log m)^2},\]
resulting from the mean value theorem. \qedhere
\end{proof}

In practice, if $\alpha$ is explicitly known (as an algebraic number in number field $L$), it is then possible to compute effectively $c_1(\alpha)$ for $m\geq 13$. For $m<13$, one just has to estimate directly $|1-\alpha^m|$.

Another way of estimating $1-\alpha^m$ is to reduce it modulo a prime ideal $\mathfrak{p}$ of $\mathcal{O}_L$. More precisely, we want to evaluate its valuation $v_{\mathfrak{p}}(1-\alpha^m)$ at $\mathfrak{p}$; for an element $z\in L$, we write $v_{\mathfrak{p}}(z)$ instead of $v_{\mathfrak{p}}(z\mathcal{O}_L)$ for more simplicity. This can be obtained as follows.

\begin{proposition}
\label{prop:vp}
Let $\alpha$ be an algebraic integer that is not a root of unity in a number field $L$ of degree $d$, and $m$ a positive integer. Let $\mathfrak{p}$ be a prime ideal of $\mathcal{O}_L$ over a prime number $p$. Assume that $\mathfrak{p}\nmid\alpha$. Denote by $m_0$ the order of $\alpha$ in $\mathcal{O}_L/\mathfrak{p}$, that is the least positive integer such that $1-\alpha^{m_0}=0\bmod\mathfrak{p}$, and $v_0=v_{\mathfrak{p}}(1-\alpha^{m_0})$. Then, assuming $p>d+1$, we have
\[v_{\mathfrak{p}}(1-\alpha^m)=
\begin{cases}
0 & \text{if }m_0\nmid m\\
sv_{\mathfrak{p}}(p)+v_0 & \text{if }m=m_0p^sr,\,\gcd(p,r)=1.
\end{cases}\]
\end{proposition}

\begin{proof}
If $m_0\nmid m$, it is clear that $1-\alpha^m\not\equiv 0\bmod\mathfrak{p}$; hence, $v_{\mathfrak{p}}(1-\alpha^m)=0$. Otherwise, write $m=m_0p^sr$ with $\gcd(p,r)=1$. We proceed by induction on $s\ge 0$. For $s=0$, factoring $1-\alpha^m$ gives
\[1-\alpha^m=(1-\alpha^{m_0})\left(\sum_{l=0}^{r-1}\alpha^{m_0l}\right).\]
Since $\alpha^{m_0l}\equiv 1\bmod\mathfrak{p}$, for all $l\in\{0,\dots,r-1\}$, we deduce
\[v_{\mathfrak{p}}(1-\alpha^m)=v_{\mathfrak{p}}(1-\alpha^{m_0})+v_{\mathfrak{p}}(r)=v_0.\]
We now let $\beta=\alpha^{rm_0}$ and treat the case $s=1$. Writing $\beta=1+\lambda$, where $\lambda\in \mathfrak{p}$, we have that
\[\frac{\beta^p-1}{\beta-1}=\frac{(1+\lambda)^p-1}{\lambda}=\sum_{k=1}^{p-1} \binom{p}{k}\lambda^{k-1}+\lambda^{p-1}.\]
In the right-hand side, we have that $v_{\mathfrak{p}}(\lambda)\ge 1$, and $v_{\mathfrak{p}}(\lambda^{p-1})\ge (p-1)>d\ge v_{\mathfrak{p}}(p)$, so
\[v_{\mathfrak{p}}\left(\sum_{k=1}^{p-1} \binom{p}{k}\lambda^{k-1}+\lambda^{p-1}\right)=v_{\mathfrak{p}}(p).\]
Hence, for $s=1$, we have
\[v_{\mathfrak{p}}(1-\alpha^m)=v_{\mathfrak{p}}(1-\alpha^{m_0r})+v_{\mathfrak{p}}\left(\frac{\beta^p-1}{\beta-1}\right)=v_0+v_{\mathfrak{p}}(p).\]
The statement now follows by induction on $s$, where the induction step from $s$ to $s+1$ is done as above (by replacing $\alpha$ by $\alpha^{p^s}$).  
\end{proof}

\section{Proof of Theorem \ref{th:main}}

Let $x=j(\tau),y=j(\tau')$ be two singular moduli of respective discriminants $\Delta$ and $\Delta'$, with $\{\Delta,\Delta'\}=\{-4\cdot 23,-23\}$ or $\{\Delta,\Delta'\}=\{-4\cdot 31,-31\}$, such that
\begin{equation}
\label{eq:main}
Ax^m+By^n=C
\end{equation}
for some $A,B,C\in\mathbb{Q}^{\times}$ and $m,n$ positive integers.

Both $x$ and $y$ are of degree $3$ over $\mathbb{Q}$, and admit one real conjugate corresponding to the dominant $j$-value, and two complex conjugates. If $x$ is real, then $y$ is also real. Indeed, if not, then, together with \eqref{eq:main}, we have
\[Ax^m+B\bar{y}^n=C.\]
We obtain that $y^n=\bar{y}^n$, which contradicts Lemma \ref{lemma:power-j}.

The equation \eqref{eq:main} implies that $\mathbb{Q}(x^m)=\mathbb{Q}(y^n)$; hence, $\mathbb{Q}(x)=\mathbb{Q}(y)$ by Lemma \ref{lemma:power-j}. In particular, the Galois orbit of $(x,y)$ over $\mathbb{Q}$ has exactly $3$ elements, and each conjugate of $x$ occurs exactly once as the first coordinate of a point in the orbit, just as each conjugate of $y$ occurs exactly once as the second coordinate.

We denote by $(x_1,y_1),(x_2,y_2),(x_3,y_3)$ the conjugates of $(x,y)$, with $x_1,y_1$ real, and $x_2,x_3$, respectively $y_2,y_3$, are complex conjugates. By \eqref{eq:main} again, the points $(x_i^m,y_i^n)$, $i\in\{1,2,3\}$, are collinear. We can write the relation of collinearity of these points in one of the following two ways:
\begin{gather}
\label{eq:colin-det}
\left|
\begin{matrix}
1 & x_1^m & y_1^n \\ 
1 & x_2^m & y_2^n \\ 
1 & x_3^m & y_3^n
\end{matrix} 
\right|=0;
\\
\label{eq:colin}
\left(\frac{x_1}{x_2}\right)^{-m}\left(\frac{y_1}{y_2}\right)^n=\frac{1-\left(\frac{y_3}{y_2}\right)^n-\left(\frac{x_3}{x_1}\right)^m}{1-\left(\frac{y_3}{y_1}\right)^n-\left(\frac{x_3}{x_2}\right)^m}.
\end{gather}

We focus first on the case $\{\Delta,\Delta'\}=\{-4\cdot 23,-23\}$, and we detail afterwards the  slight differences in the treatment of the case $\{\Delta,\Delta'\}=\{-4\cdot 31,-31\}$. We denote by $L$ the Galois closure of $\mathbb{Q}(x)=\mathbb{Q}(y)$, which by definition contains all $x_i$'s and $y_i$'s.

As announced above, we consider the case $\Delta=4\Delta'=-4\cdot 23$.

Using \textsf{PARI}, one can find a prime ideal $\mathfrak{p}$ of $\mathcal{O}_L$ over $p=23$ such that $\mathfrak{p}|x_2\mathcal{O}_L$, $\mathfrak{p}|x_3\mathcal{O}_L$, but $\mathfrak{p}\nmid x_1y_2y_3\mathcal{O}_L$. Hence, modulo $\mathfrak{p}^m$, the equation \eqref{eq:colin-det} becomes
\[1-\alpha^n=0\bmod\mathfrak{p}^m,\]
with $\alpha=y_3/y_2$. On the one hand, we deduce that $m\leq v_{\mathfrak{p}}(1-\alpha^n)$. On the other hand, we apply Proposition \ref{prop:vp}, checking first that $1-\alpha=0\bmod\mathfrak{p}$, $v_{\mathfrak{p}}(1-\alpha)=1$, $v_{\mathfrak{p}}(p)=2<6<22=p-1$; writing $m=p^sr$ with $\gcd(p,r)=1$, we get
\[v_{\mathfrak{p}}(1-\alpha^m)=sv_{\mathfrak{p}}(p)+1=2s+1.\]
Consequently,
\begin{equation}
\label{eq:bornmn-padic}
m\leq 2\frac{\log n}{\log 23}+1.
\end{equation}

Next, we want to estimate the expression on the  right--hand side of \eqref{eq:colin} in terms of $m$ and $n$ (in fact, only in terms of $n$ thanks to \eqref{eq:bornmn-padic}), in order to obtain a bound on $n$. The principal difficulty is to find a lower bound of the absolute value of its denominator. Since $y_3/y_1$ is pretty close to $0$, it depends essentially on the quantity $1-\beta^m$ with $\beta=x_3/x_2$. Noticing that $|\beta|=1$ and $\beta$ is not a root of unity, then according to Theorem \ref{th:mignotte}, there exists a constant $c_1(\beta)>0$ such that
\[|1-\beta^m|>0.99\mathrm{e}^{-c_1(\beta)(\log m)^2}.\]
Explicitly, for $m\geq 13$, we can choose $c_1(\beta)=4973.14$. It follows that
\begin{align*}
\left|1-\left(\frac{y_3}{y_1}\right)^n-\left(\frac{x_3}{x_2}\right)^m\right| &> 0.99\mathrm{e}^{-4973.15(\log m)^2}-\left|\frac{y_3}{y_1}\right|^n\\
&> 0.99\mathrm{e}^{-4973.14\left(\log\left(2\frac{\log n}{\log 23}+1\right)\right)^2}-\left|\frac{y_3}{y_1}\right|^n
\end{align*}
(recall the inequality \eqref{eq:bornmn-padic}). By a quick calculation, we observe that the last term of the previous inequality is positive provided that $n>2074$. More specifically, if $n>2075$, then
\[\left|1-\left(\frac{y_3}{y_1}\right)^n-\left(\frac{x_3}{x_2}\right)^m\right|>0.98\mathrm{e}^{-4973.14\left(\log\left(2\frac{\log n}{\log 23}+1\right)\right)^2}.\]
Finally, for $m\geq 13$ and $n>2075$, we have
\[\left|\frac{x_1}{x_2}\right|^{-m}\left|\frac{y_1}{y_2}\right|^n\leq 2.05\mathrm{e}^{4973.14\left(\log\left(2\frac{\log n}{\log 23}+1\right)\right)^2},\]
and
\begin{align*}
-\left(2\frac{\log n}{\log 23}+1\right)\log\left|\frac{x_1}{x_2}\right|+n\log\left|\frac{y_1}{y_2}\right|\leq &~ \log 2.05\\
& +4973.14\left(\log\left(2\frac{\log n}{\log 23}+1\right)\right)^2.
\end{align*}
This last inequality yields $n\leq 2092$, and then \eqref{eq:bornmn-padic} gives $m\le 5$. This is in contradiction with the previous assumptions $m\geq 13$ and $n>2075$. Therefore, either $m<13$ or $n\leq 2075$. In both cases, $m<13$, and for each possible $m$, we can explicitly compute a constant $c_2(m)$ such that
\[\left|\frac{x_1}{x_2}\right|^{-m}\left|\frac{y_1}{y_2}\right|^n\leq c_2(m).\]
This allows to bound $n$. The table below summarizes all constants $c_2(m)$ and all bounds we obtain.
\begin{table}[H]
\caption{Constants $c_2(m)$ and bounds on $n$ for each $m<13$, in the case $\Delta=4\Delta'=-4\cdot 23$}
\[
\begin{array}{c|c|c}
m & c_2(m) & \textrm{Upper bound of }n\\ \hline
1 & 1.15 & 2\\ \hline
2 & 1.21 & 5\\ \hline
3 & 11.97 & 8\\ \hline
4 & 1.10 & 10\\ \hline
5 & 1.28 & 13\\ \hline
6 & 6.00 & 16\\ \hline
7 & 1.07 & 18\\ \hline
8 & 1.38 & 21\\ \hline
9 & 4.02 & 24\\ \hline
10 & 1.04 & 26\\ \hline
11 & 1.50 & 29\\ \hline
12 & 3.04 & 32
\end{array}
\label{table:allmn}
\]
\end{table}
Again,  inequality \eqref{eq:bornmn-padic} eliminates all entries of Table \ref{table:allmn} with $m\geq 3$. Consequently, either $m=1$ and $n\leq 2$, or $m=2$ and $n\leq 5$. For each of these remaining couples $(m,n)$, a direct calculation shows that the determinant in equation \eqref{eq:colin-det} does not vanish.

To finish, we repeat this process for the case $\Delta=4\Delta'=-4\cdot 31$. In this case, one can find a prime ideal $\mathfrak{p}$ of $\mathcal{O}_L$ over $p=11$ such that $\mathfrak{p}|x_2\mathcal{O}_L$, $\mathfrak{p}|x_3\mathcal{O}_L$, but $\mathfrak{p}\nmid x_1y_2y_3\mathcal{O}_L$ as before, and we get
\begin{equation}
\label{eq:bornmn-padic-bis}
m\leq\frac{\log n}{\log 11}+2.
\end{equation}
We obtain as well, for $m\geq 13$ and $n>1440$,
\[\left|\frac{x_1}{x_2}\right|^{-m}\left|\frac{y_1}{y_2}\right|^n\leq 2.05\mathrm{e}^{4820.16\left(\log\left(\frac{\log n}{\log 11}+2\right)\right)^2},\]
then
\[-\left(\frac{\log n}{\log 11}+2\right)\log\left|\frac{x_1}{x_2}\right|+n\log\left|\frac{y_1}{y_2}\right|\leq\log 2.05+4820.16\left(\log\left(\frac{\log n}{\log 11}+2\right)\right)^2,\]
which yields $n\leq 1720$ and $m\leq 5$; again a contradiction. For each possible $m<13$, we compute a constant $c_2(m)$ as defined above, and we deduce a bound on $n$. Here is the table:
\begin{table}[H]
\caption{Constants $c_2(m)$ and bounds on $n$ for each $m<13$, in the case $\Delta=4\Delta'=-4\cdot 31$}
\[
\begin{array}{c|c|c}
m & c_2(m) & \textrm{Upper bound of }n\\ \hline
1 & 1.13 & 3\\ \hline
2 & 1.25 & 6\\ \hline
3 & 6.17 & 10\\ \hline
4 & 1.06 & 13\\ \hline
5 & 1.44 & 16\\ \hline
6 & 3.13 & 19\\ \hline
7 & 1.02 & 22\\ \hline
8 & 1.76 & 26\\ \hline
9 & 2.13 & 29\\ \hline
10 & 1.01 & 32\\ \hline
11 & 2.33 & 36\\ \hline
12 & 1.65 & 39
\end{array}
\]
\label{table:allmn-bis}
\end{table}
Inequality \eqref{eq:bornmn-padic-bis} eliminates all entries of Table \ref{table:allmn-bis} with $m\geq 3$. Consequently, either $m=1$ and $n\leq 3$, or $m=2$ and $n\leq 6$. Each of these remaining possibilities can be excluded by a direct calculation showing that the respective determinant does not vanish.


\end{document}